\documentclass[12pt]{amsart}

\usepackage[utf8]{inputenc}
\usepackage[T1]{fontenc}

\usepackage[dvipsnames]{xcolor}
\usepackage{amsmath, amsfonts, amsthm, amssymb}
\usepackage[top=1in, bottom=1in, left=1in, right=1in]{geometry}
\usepackage[backref=page, colorlinks=true, allcolors=blue]{hyperref}
\usepackage[alphabetic,lite,backrefs]{amsrefs} 
\usepackage{tikz}
\usepackage{changepage}
\usepackage{mathrsfs}
\usepackage{tikz-cd}
\usepackage{mathtools} 
\usepackage{bm} 
\usepackage{comment} 
 
\usepackage{enumitem} 

\usepackage[colorinlistoftodos,prependcaption,textsize=scriptsize, obeyFinal]{todonotes}

\makeatletter
\providecommand\@dotsep{5}
\renewcommand{\listoftodos}[1][\@todonotes@todolistname]{%
  \@starttoc{tdo}{#1}}
\makeatother

\theoremstyle{plain}
	\newtheorem{theorem}{Theorem}[section]
	\newtheorem{lemma}[theorem]{Lemma}
    \newtheorem{corollary}[theorem]{Corollary}

\theoremstyle{definition}
    
    \newtheorem*{defn*}{Definition}
    \newtheorem{example}[theorem]{Example}
    \newtheorem*{example*}{Example}
    
\theoremstyle{remark}
	
	\newtheorem*{remark*}{Remark}

\usepackage{mathtools}
\newcommand\SetSymbol[1][]{\nonscript\:#1\vert\allowbreak\nonscript\:\mathopen{}}
\providecommand\given{} 
\DeclarePairedDelimiterX\Set[1]\{\}{\renewcommand\given{\SetSymbol[\delimsize]}#1}

\AtBeginEnvironment{example}{%
  \pushQED{\qed}%
}
\AtEndEnvironment{example}{\popQED\endexample}

\def\zz{{\mathbb{Z}}}
\def\nn{{\mathbb{N}}}
\def\qq{{\mathbb{Q}}}

\def\Z{{\mathbb{Z}}}

\DeclareMathOperator{\pdim}{pdim}

\renewcommand{\mod}{\text{ mod }}

\title{Semigroup Graded Stillman's Conjecture}

\author{John Cobb}
\address{Department of Mathematics, Auburn University, Auburn, AL}
\email{jdcobb3@gmail.com}
\urladdr{\href{https://johndcobb.github.io}{https://johndcobb.github.io}}

\author{Nathaniel Gallup}
\address{Department of Mathematics, Northeastern University Oakland Campus, Oakland, CA}
\email{n.gallup@northeastern.edu}
\urladdr{\href{https://sites.google.com/view/nathanielgallup/home}{https://sites.google.com/view/nathanielgallup/home}}

\author{John Spoerl}
\address{Department of Mathematics, University of Wisconsin, Madison, WI}
\email{john.spoerl@wisc.edu}

\begin{document}

\maketitle
\vspace{-0.3in}
\begin{abstract}
    We resolve Stillman's conjecture for families of polynomial rings that are graded by any abelian group under mild conditions. Conversely, we show that these conditions are necessary for the existence of a Stillman bound. This has applications even for the well-known standard graded case.
\end{abstract}


\section{Introduction}    

Let $k$ be a field and $S=k[x_1,\dots, x_n]$ be the $n$-variable polynomial ring. Famously, the Hilbert syzygy theorem says that any finitely generated $S$-module has projective dimension less than $n$. In 2000, Mike Stillman conjectured that the projective dimension of a homogeneous ideal in a polynomial ring can be bounded just in terms of the number and degrees of the generators provided that $S$ is given the standard grading \cite{PeevaStillman2009}. Importantly, this bound is independent of the number of variables of $S$. This conjecture was proven by Ananyan and Hochster \cite[Theorem C]{AnanyanHochster2016} in 2016, and subsequently reproven by Erman, Sam, and Snowden \cite{ErmanSamSnowden19} and Draisma, Laso\'n, and Leykin \cite{DraismaLasonLeykin19}.

The property that the projective dimension of finitely generated ideals are bounded only in terms of the number and degrees of their generators is sometimes called being \textit{Stillman bounded} and is part of a larger program of similar phenomena of Stillman uniformity \cite{ErmanSamSnowden18}. For instance, Caviglia proved that a Stillman bound on projective dimension is equivalent to one on regularity \cite{peeva2010graded}*{Theorem 29.5} which links to work on the Eisenbud--Goto conjecture \cite{BederMcCullough2011, mccullough2018counterexamples}. More generally, ``projective dimension'' can be swapped out with a host of other ideal invariants \cite{erman2021generalizations}. Due mostly to work in toric geometry, there has been great interest in understanding analogs of such results for gradings by other abelian groups  \cites{HaimanSturmfelds2002, MaclaganSmith04, HeringSchenckSmith06, Costa2006mBlocksCA, Huy07, lozovanu2012vanishing, berkesch2017virtual, Yang19, ChardinNemati20, BrownSayrafi22, chardin2022multigraded, Bruce2022BoundsOM,booms2022virtual, cobb2023syzygies}. Since Stillman uniformity requires finding a bound independent of the number of variables, it is inherently a property of a \textit{family} of graded polynomial rings (e.g. all standard graded polynomial rings) and not of any particular fixed $S$. Ananyan and Hochster established that the family of  $\mathbb{Z}_+$-graded polynomial rings has Stillman bounded projective dimension \cite[page 12]{AnanyanHochster2016} for possibly inhomogeneous ideals. It is our goal to understand which graded families have Stillman uniformity.

The following example demonstrates that allowing the grading group to have infinite descending chains (i.e. the divisibility order is not well-founded) can prevent Stillman bounds on projective dimension. 

\begin{example}\label{ex: no bounded factorization}
    Let $S_n=k[x,y,z_{1},\dots, z_{n}]$ be a $\mathbb{Q}_+$-graded polynomial ring with grading given by $\deg(x)=\deg(y)=\deg(z_i)=1/n$. The homogeneous ideal
    \begin{equation*}
        I_n = \left\langle x^n, y^n, x^{n-1}z_1+x^{n-2}yz_2+\dots + xy^{n-2}z_{n-1} + y^{n-1}z_n\right\rangle
    \end{equation*}
    is generated by $3$ degree $1$ elements. It is shown in \cite{mccullough2011family} that the projective dimension of $S_n/I_n$ is $n+2$. Therefore the family $\{ S_n \}$ of $\mathbb{Q}_+$-graded polynomial rings cannot have a Stillman bound since we can make $n$ arbitrarily large while $I_n$ is generated in degree $1$. 
\end{example}

On the other hand, there is an obvious condition on the grading which is sufficient for the existence of a Stillman bound. If there is a sufficiently nice ``flattening'' map from the grading group to $\mathbb{Z}$, we can regrade our family by $\mathbb{Z}_+$ and apply Ananyan and Hochster's result, since changing the grading does not change projective dimension. However, not all grading groups admit such a map (see Example \ref{ex: no flattening map}). Our main result gives a condition on the grading group that is weaker than having a flattening map and stronger than well-foundedness which is equivalent to the existence of a Stillman bound for any graded family. To be more precise, we need a few definitions.

A grading of a polynomial ring $S = k[X]$ over a (possibly infinite) set of variables $X$ by a (possibly infinitely generated) abelian group $\Gamma$ is a decomposition of $S$ into $k$-submodules
\[ S = \bigoplus_{g\in \Gamma} S_g, \hspace{0.2in} \text{with } S_g \cdot S_h \subset S_{g+h}\] such that for all variables $x \in X$, $x \in S_g$ for some $g \in \Gamma$. If $s \in S_g$ we say that $s$ is \emph{homogeneous of degree $g$} and write $\deg(s) = g$. We say that $S$ is \textit{connected} if $S_0 = k$.
The \emph{(grading) support} of $S$ is the submonoid of $\Gamma$ generated by the degrees of all the monomials along with the identity. If $S$ is connected, this submonoid inherits a partial order to be described in (\ref{eq:order}). We say a (possibly inhomogeneous) ideal has degree sequence \textit{bounded by} $\bm{d} =(d_1,\dots, d_n)\in \Gamma^n$ if the ideal can be generated by elements $f_1,\dots, f_n$ whose monomials have degree less than $d_1,\dots, d_n$. Any submonoid $\Lambda$ of the grading group $\Gamma$ has \emph{bounded factorization} if it is impossible to express an element in $\Lambda$ as an arbitrarily large sum of other elements in $\Lambda$. 
Our main theorem shows that for a fixed $\Lambda \subseteq \Gamma$, any family of connected $\Gamma$-graded polynomial rings with support contained in $\Lambda$ has Stillman bounded projective dimension if and only if $\Lambda$ has bounded factorization. 

\begin{theorem}\label{thm:polynomial_stillman}
For any degree sequence $\bm{d}$ from $\Lambda$, there is a number $N(\Lambda, \bm{d})$ depending only on $\Lambda$ and $\bm{d}$ bounding the projective dimension of any ideal with degree sequence bounded by $\bm{d}$ in any connected $\Gamma$-graded polynomial ring with support contained in $\Lambda$ if and only if $\Lambda$ has bounded factorization.
\end{theorem}

After fixing any particular polynomial ring $S$ and ideal $I$, the projective dimension of $S/I$ does not change as one varies the grading. Hilbert's syzygy theorem guarantees that the projective dimension of \textit{every} ideal in $S$ is bounded by the (hopefully finite) number of variables in $S$. The classical Stillman's conjecture ensures a bound shared across the entire family of standard graded polynomial rings at the cost of restricting the desired class of ideals $I$ to those with a particular degree sequence. Theorem \ref{thm:polynomial_stillman} asserts that if you also allow your grading to vary, under mild conditions you must only further fix $\Lambda$ for Stillman uniformity. Example \ref{ex: no bounded factorization} fails to have Stillman bounded projective dimension exactly because the supports of the family fail to lie in a bounded factorization monoid. Letting $\deg(x)=\deg(y)=0$ and $\deg(z_i)=d$ in the same example shows that non-connected gradings also fail to have a Stillman bound.

After finding the right condition on $\Lambda$, the proof of Theorem \ref{thm:polynomial_stillman} ends up being very short. The key intuition is to leverage the known Stillman bound for the $\Z_+$-graded case by using bounded factorization to guarantee a nice regrading by $\Z_+$ for the ``if'' direction, and to construct an explicit counterexample using only non-bounded factorization for the ``only if'' direction. Even in the $\Lambda=\Z_+$ case proven by Ananyan and Hochster \cite{AnanyanHochster2016}, Theorem \ref{thm:polynomial_stillman} says something new; the same Stillman bound holds for families of varying $\Z_+$-gradings, with the bounds potentially getting tighter for polynomial rings whose support refines $\Z_+$. We end by discussing some other applications and further examples in Section \ref{sec: Applications and Examples}.

\renewcommand*{\thetheorem}{\thesection.\arabic{theorem}}

\section{$\Gamma$-graded Commutative Algebra}\label{sec:G-Graded Commutative Algebra}

Let $\Gamma$ be an abelian group, $k$ a field, and $S$ a $\Gamma$-graded polynomial ring with graded support $\Lambda$. The graded structure (and thus many of the algebraic properties) of $S$ is entirely determined by $\Lambda$. In fact, we could eliminate mentions of $\Gamma$ as long as we assume that $\Lambda$ is commutative and  \textit{cancellative}, which means that $g+h=g+h'$ implies $h=h'$ for all $g,h,h' \in \Lambda$. This comes for free when $\Lambda$ is a submonoid of a group, otherwise, it allows us to construct a unique group (its Grothendieck group) containing $\Lambda$ by adding formal inverses. We will nearly always assume that $\Lambda$ is \textit{pointed}, which means that $q+q'=0$ implies $q=q'=0$ for $q,q' \in \Lambda$. Note that $\Lambda$ must be torsion-free in order to be pointed. If $\Lambda$ is pointed and cancellative, it has a natural partial order $\leq_\Lambda$:
\begin{equation} \label{eq:order}
g \leq_\Lambda h \iff g + q =h \text{ for some } q \in \Lambda. \end{equation}
This provides a notion of positivity in $S$ by letting $S_+ = \Set*{ s \in S \given \deg(s) >_\Lambda 0}$ be the \textit{positive degree} elements of $S$. By assuming that $\Lambda$ is commutative and cancellative we can work directly with $\Lambda$-graded rings without mention of $\Gamma$. The following lemma shows that for $\Gamma$-graded polynomial rings, connected implies pointed:
\begin{lemma}\label{lem: connected implies pointed for polynomial rings}
If $S$ is a connected $\Lambda$-graded polynomial ring then the support of $S$ is pointed. 
\end{lemma}
\begin{proof}
Let $S = k[X]$ and suppose that $g , h \in \Lambda$ are such that $g + h = 0$. Then by definition of $\Lambda$ there exists $r , s \in S \smallsetminus \{ 0 \}$ such that $\deg(r) = g$ and $\deg(s) = h$. Then $\deg(rs) = \deg(r) + \deg(s) = g + h = 0$, so by connectedness, we have that $rs \in k$. Since $S$ is a polynomial ring over $k$, we have $r \in k$ and $s \in k$, hence $g = \deg(r) = 0$ and $h = \deg(s) = 0$.
\end{proof}

 \noindent We say that $\Lambda$ (or $S$):
\begin{itemize}
    \item is \textit{well-founded} if there are no infinite decreasing chains in $\Lambda$ under the relation $\leq_\Lambda$,
    \item has \emph{bounded factorization} if for all $g \in \Lambda$ there exists some $N$ such that for every expression $g = g_1 + \ldots + g_s$ for $g_1, \ldots, g_s \in \Lambda \smallsetminus \{ 0 \}$, we have $s \leq N$,
\end{itemize}

If $\Lambda$ is pointed, having bounded factorization allows a well-defined height function that assigns any element $\bm{d}$ in $\Lambda$ to a natural number $|\bm{d}|$ satisfying a triangle inequality (see \cite{gotti2023atomic}). Although bounded factorization implies well-foundedness, the following example shows that there exist well-founded monoids which do not have bounded factorization. 

\begin{example}[Well-founded but no bounded factorization]\label{ex: well-founded but not BFM}
Let $\Lambda$ be the submonoid of $\qq_{\geq 0}$ generated by $\{ \frac{1}{p} \mid p \in \nn \text{ is prime}\}$. $\Lambda$ is pointed since it is contained in the pointed monoid $\qq_{\geq 0}$, but it does not have bounded factorization since $1$ can be written as a sum of $p$ copies of $1/p$ for all primes $p$. However, we claim that $\Lambda$ is well-founded. Since all elements in $\Lambda$ are built by adding up reciprocals of primes, any element $h$ in $\Lambda$ can be written uniquely as $h_1 + \sum_{k=2}^\ell \frac{h_k}{p_k}$, where $0< h_k < p_k$. We wish to show that any decreasing path from $h$ is finite. Suppose we start such a path, $h \geq_\Lambda g$. Then there must exist some $g' \in \Lambda$ with $g + g' = h$. Writing each of these in the form mentioned above, it must be that $g_k+g'_k = h_k$ and thus that $g_k \leq h_k$. So any chain of elements strictly descending from $h$ can have at most $h_1 + \ldots + h_\ell$ elements.
\end{example}

\newpage 
\section{Main Results and Corollaries}\label{sec:Main Results and Corollaries}

In this section, $\Lambda$ will always be commutative and cancellative. Consider a monomial of degree $d\in \Lambda$ in a connected $\Lambda$-graded polynomial ring $k[X]$. If we completely forget the $\Lambda$-grading and give $k[X]$ the standard grading (i.e. ``flatten'' it), what degree will $f$ be? The following theorem says that having bounded factorization is exactly what is needed to bound the new  degree of $f$ in the standard grading.

\begin{lemma} \label{lem: bound on the exponents of a monomial}
    Suppose $k[X]$ is a connected $\Lambda$-graded polynomial ring with bounded factorization. For all $g \in \Lambda$, there exists $N \in \mathbb{N}$ such that if $x_1^{e_1} \ldots x_m^{e_m}$ has degree $\leq_\Lambda g$ (for $x_i \in X$) then $e_1 + \ldots + e_m \leq N$. 
\end{lemma}

\begin{proof}
By definition of having bounded factorization, there exists $N \in \nn$ such that if $g_1 + \ldots + g_s = g$ is a non-trivial factorization of $g$ (i.e. $g_i \neq 0$) then $s \leq N$. If $\deg(x_1^{e_1} \ldots x_m^{e_m}) \leq_\Lambda g$ then there exists some $h \in \Lambda$ with $e_1 \deg(x_1) + \ldots + e_m \deg(x_m) + h = g$. Since $\deg(x_j) >_\Lambda 0$ by connectedness, it must be that $e_1 + \ldots + e_m \leq N$ as desired. 
\end{proof}

The forward direction of Theorem \ref{thm:polynomial_stillman} comes from constructing a specific counterexample only from the knowledge that $\Lambda$ is not a bounded factorization monoid. Consider the following example, due to Burch in the local case \cite{burch1968note}, Kohn in the global case \cite{kohn1972ideals}, and translated to the language of polynomial rings by McCullough and Seceleanu \cite{mccullough2012bounding}. 

\begin{example}\label{ex: unbounded projective dimension linear}
    Let $S = k[x_1, \ldots, x_n, y_1, \ldots, y_n]$, and let $I$ be the ideal generated by the three elements $f_1 = \prod_{i = 1}^n x_i$, $f_2 = \prod_{i = 1}^n y_i$, and $f_3 = \sum_{i = 1}^n \prod_{j \neq i} x_j y_j.$ Then $\pdim_S(S / I) = n + 2$.
\end{example}

\begin{proof}[Proof of Theorem \ref{thm:polynomial_stillman}]
    For the forward direction, suppose that $\Lambda$ does not have bounded factorization. This means there exists some $d\in \Lambda$ such that for every $b\in \nn$  we can find another number $B$ larger than $b$ that is the length of a factorization $d=d_1+\dots+ d_B$, where $d_1,\dots, d_B$ are nonzero degrees in $\Lambda$. Grouping terms arbitrarily, we can take $B=b$. For every $b$ we can define a polynomial ring $S_b=k[x_1,\dots, x_b,y_1,\dots, y_b]$ with $\Lambda$-grading defined by setting $\deg(x_i)=\deg(y_i)=d_i$. Then with $f_1, f_2, f_3$ defined as in Example \ref{ex: unbounded projective dimension linear}, we have that $f_1$ and $f_2$ are homogeneous of degree $d$, while $f_3$ is inhomogeneous but with the degree of each monomial bounded by $2d$. Letting $I_b = \langle f_1,f_2,f_3\rangle$, Example \ref{ex: unbounded projective dimension linear} confirms that $\pdim_{S_b}(S_b/I_b) = b+2$ can grow arbitrarily big, so there cannot be a Stillman bound. 
    
    For the backward direction, let $S$ be any connected $\Lambda$-graded polynomial ring with bounded factorization and let $I=\langle f_1,\dots, f_n\rangle$ be an ideal with degree sequence bounded by $\bm{d}$. By Lemma \ref{lem: bound on the exponents of a monomial}, we can find a number $B$ bounding the standard degree of all monomials appearing among $f_1,\dots, f_n$ after giving $S$ the standard grading. Now, \cite[Theorem C]{AnanyanHochster2016} provides a Stillman bound for $I$.
 \end{proof}

Corollary \ref{cor: refined support gives better bounds} shows that any connected polynomial ring with support $\Lambda'\subset \Lambda$ has the same Stillman bound, i.e. $N(\Lambda',\bm{d}) \leq N(\Lambda, \bm{d})$. 

 \begin{corollary}\label{cor: refined support gives better bounds}
    If $S'$ and $S$ have nested support $\Lambda' \subset \Lambda$ then $N(\Lambda',\bm{d}) \leq N(\Lambda, \bm{d})$.
\end{corollary}

\begin{proof}
    Given such an $S$ with support $\Lambda' \subseteq \Lambda$ and, let $Y = \{ y_g \mid g \in \Lambda \smallsetminus \Lambda' \}$ be a set of variables in bijection with $\Lambda \smallsetminus \Lambda'$. We give $T := k[X \sqcup Y]$ a $\Lambda$-grading by setting $\deg_T(x) = \deg_S(x)$ for all $x \in X$ and $\deg_T(y_g) = g$. Then $T$ is a connected $\Lambda$-graded polynomial ring with $S$ as a subring, hence the grading has bounded factorization and has the bound $N(\Lambda, \bm{d})$ as in Theorem \ref{thm:polynomial_stillman}.
\end{proof}

\section{Applications and Examples}\label{sec: Applications and Examples}

We emphasize that the bound in Theorem \ref{thm:polynomial_stillman} works over families of different gradings with the following example. 

\begin{example}\label{ex: BFM but not downward-finite}
Let $\Gamma = \zz \times \zz$ and let $\Lambda = (\zz \times \zz_{>0}) \cup \{ (0 , 0) \}$. Then despite being infinitely generated, $\Lambda$ has bounded factorization.\footnote{This is true since we have a height function $\ell : \Lambda \to \zz_{\geq 0 }$ defined by $\ell(n , m) = m$, meaning $\ell$ satisfies $\ell(g) = 0$ iff $g = 0$ and $\ell(g + h) \geq \ell(g) + \ell(h)$ for all $g , h \in \Lambda$. A pointed monoid has bounded factorization if and only if it has such a height function \cite{gotti2023atomic}.} Now for any $n$, let $S_n = k[x_1, \ldots, x_n, y_1, \ldots, y_n]$ and give $S_n$ a $\Gamma$-grading by setting $\deg(x_n) = (-n, 1)$ and $\deg(y_n) = (n , 1)$. Theorem \ref{thm:polynomial_stillman} says that the family $\{S_n \}$ has Stillman uniformity because for all $n$, the support of $S_n$ is contained in $\Lambda$ which has bounded factorization. In fact we could have assigned the degrees of $x_n$ and $y_n$ to be anything inside of $\Lambda$ and the same Stillman bounds would hold. 
\end{example}

One possible application of Theorem \ref{thm:polynomial_stillman} is that by equipping a ring with a finer-graded structure, we can get better Stillman bounds for a given ideal. We demonstrate this with the following example.

\begin{example}\label{ex: bad standard graded bound}
    Let $S=k[x_1,x_2, \dots]$ be an infinite polynomial ring, and suppose we want to bound the projective dimension of a particular ideal $I$ generated by $f=x_1x_4x_7+x_{10}x_{13}x_{16}$, $g=x_2x_5x_8+x_{11}x_{14}x_{17}$, and $h=x_3x_6x_9+x_{12}x_{15}x_{18}$. Of course, we get the bound of $18$ from Hilbert's syzygy theorem since $f$, $g$, and $h$ only involve $18$ variables, but we can do better. Under the standard grading, $\langle f, g, h \rangle$ has the degree sequence $\bm{d} = (3,3,3)$ whose tight Stillman bound has been computed to be $5$ \cite{mantero2019projective}. Consider the $\Z^3$-grading on $S$ by 
\begin{equation*}
    \deg(x_i) = \begin{cases}
        (1,0,0) & \text{if } i = 1 \mod 3,\\
        (0,1,0) & \text{if } i = 2 \mod 3,\\
        (0,0,1) & \text{if } i = 0 \mod 3.
    \end{cases}
\end{equation*}

Under this new grading, the degree sequence of $\langle f, g, h \rangle$ is $\bm{e} = \{ (3,0,0), (0,3,0),(0,0,3) \}$. This grading is \textit{finer} in the sense that the family of ideals that are degree $\bm{e}$ is strictly contained in the larger family of ideals that are degree $\varphi(\bm{e})=\bm{d}$, where $\varphi$ is the natural additive map summing the entries. Therefore one would expect that $N(\bm{e}) < N(\bm{d})$ and in this case, this is true; the Stillman bound under this grading is exactly the projective dimension of $3$. This is because if $\langle f, g, h \rangle$ has degree sequence $\bm{e}$, the variables showing up among $f$, $g$, and $h$ are pairwise disjoint and therefore $f , g , h$ forms a regular sequence.
\end{example}

In general, if the bound in Theorem \ref{thm:polynomial_stillman} were made effective (e.g. as in the standard graded case  \cite{kazhdan2020properties}), one could find better Stillman bounds for a fixed polynomial ring $S$ and ideal $I$ by varying the grading. That is, given a fixed $\Z_+$-graded $S$ and ideal with degree sequence $\bm{d}$, you may consider equipping $S$ with various positive $\Gamma$-gradings with homomorphisms $\varphi: \Lambda \to \Z_+$ to obtain the bound
\begin{equation} \label{eq: pdim bounded by family of gradings}
\pdim_k(S/I) \leq \min_{\Lambda} \Set{ N(\varphi(\Lambda), \varphi(\bm{e})) \given \Lambda \text{ grades } S \text{ s.t. } \varphi(\bm{e}) = \bm{d}}
\end{equation}
which Example \ref{ex: bad standard graded bound} shows can be tight.

If $S$ is a connected $\Gamma$-graded polynomial ring, then the following example shows how a choice of $\Lambda$ is equivalent to a choice of an effective cone in toric geometry.

\begin{example}
    The Cox ring of the Hirzebruch $\mathcal{H}_2$ is $S=k[x_0,x_1,x_2,x_3]$ with the $\Z^2$-grading given by $\deg(x_0)=\deg(x_2) = (1,0)$ and $\deg(x_1) = (-2,1)$ and $\deg(x_3) = (0,1)$. Since $S$ is connected, the effective cone $\{d \in \Z^2 \mid S_d \neq 0 \}$ of $\mathcal{H}_2$ is exactly the support $\Lambda$. Since $\Lambda$ has bounded factorization, there is a number  $N(\Lambda,\bm{d})$ bounding the projective dimension of all ideals whose degree sequence is less than $\bm{d}$ in any variety whose effective cone sits inside that of $\mathcal{H}_2$.
\end{example}

As discussed in the introduction, some $\Gamma$-gradings under consideration are exotic enough to fail to admit a homomorphism $\Lambda \to \Z_+$. The following is one such example with Stillman bounded projective dimension:

\begin{example}\label{ex: no flattening map}
Let $S = k[x_0, x_1, x_2, \ldots ]$ and define $p_n$, where $n \in \mathbb{N}$, to be the $n$th prime natural number. Define the degree of any constant to be $0$, $\deg(x_0) = 1 \in \mathbb{Q}$, and $\deg(x_n) = n + \frac{1}{p_n} \in \mathbb{Q}$. Let $\Lambda \subseteq \mathbb{Q}$ be the $\mathbb{Q}$-support of this grading and then let $\Gamma$ be the subgroup of $\mathbb{Q}$ spanned by $\Lambda$. Then the grading on $S$ is connected and thus $\Lambda$ is pointed. Given any $q \in \qq$ it is clear that $\{ n \mid n + \frac{1}{p_n} \leq q \}$ is finite, so $\Lambda$ has bounded factorization. Therefore due to Theorem \ref{thm:polynomial_stillman}, $S$ has Stillman bounded projective dimension. Interestingly, there is no homomorphism $\Lambda \to \mathbb{Z}_+$. Indeed if $\varphi$ is such a map, since $1$ and $\frac{1}{p_n}$ are contained in $\Gamma$ for all $n$, we have that $\varphi(1) = p_n \varphi(\frac{1}{p_n})$ for all $n$, which is impossible as $\varphi$ takes values in the integers. 
\end{example}

\subsection*{Open problems} 
Here are some open problems raised by our work:
\begin{itemize}
    \item Are there similar bounds for a larger class of ideal invariants such as multigraded regularity as shown in \cite{erman2021generalizations}? 
    \item When is the bound in Equation (\ref{eq: pdim bounded by family of gradings}) tight?
\end{itemize}

\subsection*{Acknowledgements} Thanks to Daniel Erman, Jason McCullough, Christopher O'Neill, and Ivan Aidun for invaluable discussions. Cobb acknowledges the support of the National Science Foundation Grant DMS-2402199.

\bibliography{bib}{}

\end{document}